\documentclass[12pt]{article}

\usepackage{amsmath, amssymb, enumerate, amsthm}
\usepackage[all]{xy}

\newcommand{\Top}{\mathbf{Top}}
\newcommand{\Set}{\mathbf{Set}}

\newcommand{\Ord}{\mathbf{Ord}}
\newcommand{\Lat}{\mathbf{Lat}}

\newcommand{\Grph}{\mathbf{Grph}}
\newcommand{\Grp}{\mathbf{Grp}}
\newcommand{\Rng}{\mathbf{Rng}}
\newcommand{\V}{\ensuremath{\mathbb{V}}}
\newcommand{\C}{\ensuremath{\mathbb{C}}}

\newcommand{\N}{\ensuremath{\mathbb{N}}}
\newcommand{\Q}{\ensuremath{\mathbb{Q}}}

\newcommand{\op}{\mathrm{op}}

\newcommand{\Sub}{\mathrm{Sub}}

\newcommand{\eot}[1]{\ensuremath{\in_{#1}}}

\newcommand{\Eq}[1]{\ensuremath{\mathrm{Eq}(#1)}}

\newcommand{\ord}[1]{\ensuremath{\{1,2,\dots,#1\}}}

\title{Characterizations of majority categories}
\author{Michael Hoefnagel}

\newtheorem{theorem}{Theorem}
\newtheorem*{theorem*}{Theorem}
\newtheorem{definition}{Definition}
\newtheorem{proposition}{Proposition}
\newtheorem{lemma}{Lemma}
\newtheorem{corollary}{Corollary}
\newtheorem{example}{Example}
\newtheorem{remark}{Remark}

\begin{document}
\maketitle

\begin{abstract}
In universal algebra, it is well known that varieties admitting a majority term admit several Mal'tsev-type characterizations. The main aim of this paper is to establish categorical counterparts of some of these characterizations for regular categories. We prove a categorical version of Bergman's Double-projection Theorem: a regular category is a majority category if and only if every subobject $S$ of a finite product $A_1 \times A_2 \times \cdots \times A_n$ is uniquely determined by its two-fold projections. We also establish a categorical counterpart of the Pairwise Chinese Remainder Theorem for algebras, and characterize regular majority categories by the classical congruence equation $\alpha \cap (\beta \circ \gamma) = (\alpha \cap \beta) \circ (\alpha \cap \gamma)$ due to A.F.~Pixley.  
\end{abstract}

\section{Introduction}\label{sec-Introduction}
The variety $\Lat$ of lattices is one which admits a \emph{majority} term, i.e., a ternary term $m(x,y,z)$ satisfying the equations $m(x,x,y) = m(x,y,x) = m(y,x,x) = x$. This property of $\Lat$ sharply distinguishes it from other familiar varieties, such as the varieties  $\Grp$ of groups, $\Rng$ of rings, and $R-\textbf{Mod}$ of modules over a ring $R$. For lattices, the \emph{median operation}:
\[
m(x,y,z) = (x\wedge y) \vee (x \wedge z) \vee (y \wedge z),
\]
as well as its lattice-theoretic dual, are both majority terms \cite{BK47}. Several theorems that hold for the variety of lattices extend to characterizations of those varieties which admit a majority term. Among such theorems is \textit{Bergman's Double-projection Theorem}, which asserts that a sublattice $S$  of a finite product of lattices $L_1 \times L_2 \times \cdots \times L_n$ is uniquely determined by its two-fold projections in $L_i \times L_j$ for $i,j = 1,2,..., n$  \cite{BerDPT}.  Also, there is the \textit{Pairwise Chinese Remainder Theorem} for lattices due to R. Willie \cite{Wil70}: if $\theta_1,\theta_2,....,\theta_n$ are congruences on a lattice $L$, and $a_1,...,a_n \in L$ any elements,  then if the system of congruences 
\[
x \equiv a_i \mod \theta_i  \tag{for $i = 1,2,...,n$}
\] 
is solvable two at a time, then it is solvable. 

A.P. Huhn showed that the theorems mentioned above extend to characterizations of those varieties which admit a majority term (for a proof see \cite{BP75}): 
\begin{theorem} \label{thm-UA-characterization}
	The following are equivalent for a variety $\V$ of algebras.
	\begin{enumerate}[(i)]
		\item $\V$ admits a majority term. 
		\item Any subalgebra $S$ of a finite product $A_1\times A_2\times \cdots \times A_n$ of algebras is uniquely determined by its two-fold projections in  $A_i \times A_j$
		\item For any congruences $\theta_1,\theta_2,....,\theta_n$  on any algebra $A$ and any elements $a_1,...,a_n \in A$, if the system of congruences 
		\[
		x \equiv a_i \mod \theta_i \tag{for $i = 1,2,...,n$}
		\] 
		is solvable two at a time, then it is solvable.
	\end{enumerate}
\end{theorem} 
These results complement the first Mal'tsev-type characterization of varieties admitting a majority term due to A.F. Pixley. 
\begin{theorem}[\cite{Pix63}] \label{thm-UA-Pixley-characterization}
	A variety $\V$ admits a majority term if and only for any algebra $X$ in $\V$, and any three congruences $\alpha,\beta, \gamma$ on $X$, we have
	\[
	\alpha \cap (\beta \circ \gamma) = (\alpha \cap \beta) \circ (\alpha \cap \gamma).
	\]
\end{theorem}
The main aim of this paper is to establish categorical counterparts of the two theorems above. Using the general techniques of \cite{ZJan04}, it is possible to reformulate statement (i) of Theorem~\ref{thm-UA-characterization} for general categories. The resulting notion, that of a \emph{majority category}, has been introduced and studied in \cite{Hoe18a}. If we are to categorically reformulate the statement of Theorem~\ref{thm-UA-Pixley-characterization}, and also the statement (ii) of the Theorem~\ref{thm-UA-characterization}, then the base category should possess a corresponding notion of composition of binary relations, as well as a corresponding notion of image-factorization of a morphism. Regular categories \cite{BGO71} provide a good context for both of these notions, and it is for regular categories that we establish the categorical counterparts of the above-mentioned theorems. 

One of the consequences of one of the main theorems (Theorem~\ref{thm-pixley-generalization}) is that a regular \emph{Mal'tsev category} (see \cite{CPP91} and \cite{CLP91})  is congruence distributive if and only if it is a majority category. This result generalizes Pixley's result for varieties, and clarifies a remark of D.~Bourn in \cite{Bou05} about whether or not the categories  $\mathrm{NReg}(\Top)$ of topological von Neumann regular rings and $\mathrm{BoRg}(\Top)$ of topological Boolean rings are fully congruence distributive or not. 
\section{Preliminaries}
Recall that a category $\C$ is said to be \emph{regular} \cite{BGO71} if: 
\begin{enumerate}[(i)]
	\item $\C$ has finite limits and coequalizers of kernel pairs. 
	\item The class of regular epimorphisms in $\C$ is pullback stable, i.e., if the diagram
	\[
	\xymatrix{
		\bullet \ar[r] \ar[d]_p & \bullet \ar[d]^f\\
		\bullet \ar[r] & \bullet 
	}
	\]
	is a pullback in $\C$, and $f$ is a regular epimorphism, then so is $p$. 
\end{enumerate}
The following are important consequences (i) and (ii), and will be used without mention in what follows:
\begin{enumerate}
	\item Every morphism $f$ in $\C$ factors as $f = me$ where $e:X \rightarrow  Q$ is a regular epimorphism and $m:Q \rightarrow Y$ a monomorphism. The factorization $f = me$ is sometimes called an \emph{image factorization} of $f$. 
	\item If $f:X \rightarrow Y$ and $f':X' \rightarrow Y'$ are regular epimorphisms in $\C$, then $f \times f':X \times X' \rightarrow Y \times Y'$ is a regular epimorphism in $\C$. 
	\item If the composite $g \circ f$ of two morphisms $f: X \rightarrow Y$ and $g:Y \rightarrow Z$ in $\C$ is a regular epimorphism, then $g$ is a regular epimorphism.  
\end{enumerate}
If $m:M_0 \rightarrow X$ and $n:N_0 \rightarrow X$ are monomorphisms in a category $\C$, then we write $m \leqslant n$ if $m$ factors through $n$, i.e., if there exists $\phi:M_0 \rightarrow N_0$ such that $n\phi = m$. This defines a preorder $\mathcal{M}(X)$ on the class of all monomorphisms in $\C$ with codomain $X$. The posetal reflection of $\mathcal{M}(X)$ is called the \emph{poset of subobjects} of $X$, and is denoted by $\Sub(X)$. Explicitly, a \emph{subobject} $S \in \Sub(X)$ is an equivalence class of monomorphisms with codomain $X$, where two monomorphisms $n,m \in \mathcal{M}(X)$ are equivalent if and only if $n \leqslant m$ and $m \leqslant n$. If $s: S_0 \rightarrow X$ is a member of $S$, then we will say that $S$ is the \textit{subobject represented} by $s$ in what follows.  

In any category $\C$ the pullback of a monomorphism along any morphism is again a monomorphism, which is to say that if the diagram:
\[
\xymatrix{
	\bullet \ar[d]_n \ar[r] & \bullet \ar[d]^m \\
	\bullet \ar[r]_f & \bullet 
}
\]
is a pullback diagram in $\C$, and $m$ is a monomorphism, then so is $n$. Given that $\C$ has pullbacks of monomorphisms along monomorphisms and $A,B \in \Sub(X)$ are any subobjects represented by $a:A_0 \rightarrow X$ and $b:B_0 \rightarrow X$ respectively, then we write $A \cap B$ for the subobject of $X$ represented by the diagonal monomorphism in any pullback 
\[
\xymatrix{
	(A\cap B)_0\ar[d]_-{p_2} \ar[r]^-{p_1} \ar[dr]& A_0 \ar[d]^{a} \\
	B_0 \ar[r]_{b} & X
}
\]
\begin{remark} \label{rem-uniqueness-of-images}
	If $f = me$ and $f = m'e'$ are two image factorizations of a morphism $f:X \rightarrow Y$ in a regular category $\C$, then $m$ and $m'$ represent the same subobject of $Y$, which is denoted by $f(X)$. Given a subobject $A$ of $X$ represented by $a:A_0 \rightarrow X$ we will write $f(A)$ for the subobject represented by the mono part of an image factorization of $fa$. Also, we will often refer to $f(A)$ as the image of $A$ under $f$.
\end{remark}
\begin{definition} \label{def-factors-through}
	Given a subobject $A \in \Sub(X)$, represented by $a:A_0 \rightarrow X$, then for any morphism $x$ with codomain $X$ we write $x \eot{S} A$ if $x$ factors through $a$, and $x$ has domain $S$. 
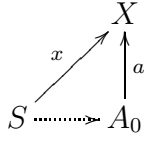
	\[
	\xymatrix{
		& X \\
		S \ar@{..>}[r] \ar[ru]^{x} & A_0 \ar[u]_a
	}
	\]
\end{definition}

\begin{remark} \label{rem-reg-eot}
	If $x$ factors through one representative of $A$, then it factors through all representatives of $A$. If $\alpha: Q \rightarrow S$ is a regular epimorphism, then $x\alpha \eot{Q} A$ if and only if $x \eot{S} A$. 
\end{remark}
\begin{definition}
	Let $\C$ be a category with binary products, then an $n$-ary relation $R$ between  $X_1,X_2,...,X_n$ is simply a subobject of $X_1 \times X_2 \times \cdots \times X_n$. 
\end{definition}
\begin{definition}[\cite{Hoe18a}] \label{def-majority-selecting}
	A ternary relation $R$ between objects $X,Y,Z$ in a category with products $\C$ is said to be \emph{majority-selecting} if it satisfies:
	\[
	(x,y,z') \eot{S} R \quad \text{and} \quad (x,y',z) \eot{S} R \quad  \text{and} \quad (x',y,z) \eot{S} R  \implies (x,y,z) \eot{S} R.
	\]
	A category $\C$ (with products) is then a \emph{majority category} if every internal ternary relation in $\C$ is majority selecting.
\end{definition}
Let $R$ be a relation between objects $X$ and $Y$, and $S$ a relation between objects $Y$ and $Z$ in a regular category $\C$. Suppose that $r = (r_1,r_2):R_0 \rightarrow X \times Y$ and $s = (s_1,s_2):S_0 \rightarrow Y \times Z$ represent $R$ ans $S$ respectively. Consider the diagram:
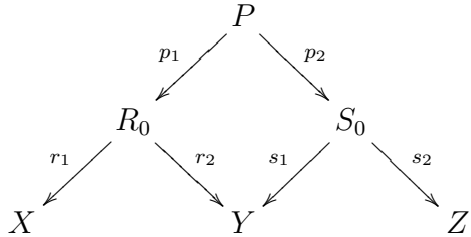
\[
\xymatrix{
	& &P \ar[dl]_{p_1} \ar[dr]^{p_2} & & \\
	& R_0 \ar[dr]^{r_2} \ar[dl]_{r_1} & & S_0 \ar[dl]_{s_1}\ar[dr]^{s_2} & \\
	X & & Y & & Z
}
\]
where $(P, p_1, p_2)$ is a pullback of $s_1$ along $r_2$. The composite $R \circ S$ is the relation represented by the monomorphism $r \circ s: (R\circ S)_0 \rightarrowtail X \times Z$, which is obtained by taking the image factorization of $(r_1p_1, s_2p_2): P \rightarrow X \times Z$ as in the diagram:
\[
\xymatrix{
	P \ar@/_1pc/[rr]_-{(r_1p_1, s_2p_2)} \ar@{->>}[r]^-e &  (R\circ S)_0 \ar[r]^-{r \circ s} & X \times Z
}
\]
\begin{proposition} \label{prop-relation-composition}
	If $(x,z): A \rightarrow X \times Z$ is any morphism, then $(x,z) \eot{A} R\circ S$ if and only if there exists a regular epimorphism $\alpha: Q \rightarrow A$ and a $y:Q \rightarrow Y$ such that $(x\alpha, y)\eot{Q} R$ and $(y, z\alpha)\eot{Q} S$. 
\end{proposition}
\begin{proof}
	If $(x,z)$ factors though $(R \circ S)_0$, then the dotted arrow $h$ exists  making the triangle in the diagram
	
	\[
	\xymatrix{
		Q\ar[d]_q \ar@{->>}[r]^-{\alpha} & A \ar[dr]^{(x,z)} \ar@{..>}[d]^h \\
		P \ar@/_1pc/[rr]_-{(r_1p_1, s_2p_2)} \ar@{->>}[r]^-e &  (R \circ S)_0 \ar[r]^-{r \circ s} & X \times Z
	}
	\]
	commute. Then we can pull $h$ back along $e$, to produce $\alpha$ and $q$ in the diagram above. Then setting $y = r_2p_1q$, we have that $\alpha$ and $y$ satisfy the required conditions.
	
	For the 'only if' part, suppose that $(x\alpha, y) \eot{Q} R$ and $(y, z \alpha) \eot{Q} S$, then it is easy to see that $(x\alpha,z\alpha) \eot{Q} R \circ S$ which by Remark~\ref{rem-reg-eot} implies that $(x,z) \eot{A} R \circ S$. 
\end{proof}
\section{Pixley's theorem for categories}
The first Mal'tsev type characterization of varieties admitting a majority term was given by Pixley (Theorem~\ref{thm-UA-Pixley-characterization}), and it states that a variety $\V$ admits a majority term if and only if for any three congruences $\alpha, \beta, \gamma$ on any algebra $A$ in $\V$ we have
	\[
	\alpha \cap (\beta \circ \gamma) = (\alpha \cap \beta) \circ (\alpha \cap \gamma).
	\]
The aim of this section is to establish the corresponding categorical  theorem.  
\begin{lemma} \label{lem-majority-implies-reflexive-relation-characterization}
	Let $\C$ be a regular majority category, then for any three reflexive relations $A,B,C$ on any object $X$ in $\C$ we have:
	\[
	(A \circ B) \cap (A \circ C) \leqslant  A \circ (B \cap C). 
	\]
\end{lemma}

\begin{proof}
	Let $a = (a_1,a_2): A_0 \rightarrow X \times X$, $b = (b_1,b_2): B_0 \rightarrow X \times X$ and $c = (c_1, c_2): C_0 \rightarrow X \times X$ represent the three reflexive relations $A,B,C$ respectively. Consider the quaternary relation $R \leqslant X^4$ represented by $r$, which is formed from the following pullback:
	\[
	\xymatrix{
		R_0 \ar[d]_r \ar[rrrr]^{p_2}& & & & A_0 \times B_0 \times C_0 \ar[d]^{a \times b \times c} \\
		X^4 \ar[rrrr]_-{((\pi_2,\pi_1), (\pi_1, \pi_3), (\pi_1, \pi_4))} & & & & (X \times X) \times (X \times X) \times (X \times X).
	}
	\]
	\noindent
	Set theoretically, $R$ is the relation defined by:
	\[
	R = \{(x,y,z,w) \in X^4 \mid (y,x) \in A \wedge (x,z) \in B \wedge (x,w) \in C \}.
	\]
	Consider the image factorization $er'$ of $(\pi_2,\pi_3,\pi_4)r$ in the diagram below:
	\[
	\xymatrix{
		R_0 \ar[r]^r \ar@{->>}[d]_e & X^4\ar[d]^{(\pi_2,\pi_3,\pi_4)} \\
		R_0'\ar[r]_{r'} & X^3
	}
	\]
	Now, let $(x,y):S \rightarrow X\times X$ be such that $(x,y) \eot{S} (A \circ B) \cap (A \circ C)$ then there exist regular epis $\alpha_1:Q_1 \rightarrow S$ and $\alpha_2:Q_2 \rightarrow S$ as well as morphisms $z_1:Q_1 \rightarrow X$ and $z_2: Q_2 \rightarrow X$ such that $(x \alpha_1,z_1) \eot{Q_1} A$ and $(z_1,y\alpha_1) \eot{Q_1} B$, together with $(x \alpha_2,z_2) \eot{Q_2} A$ and $(z_2,y\alpha_2) \eot{Q_2} C$. We may assume that $\alpha_1 = \alpha = \alpha_2$,  since if not, we could pullback $\alpha_1$ along $\alpha_2$. Then note that we have that
	\[
	(z_1,x\alpha,y\alpha,z_1) \eot{Q} R \quad \text{ and } \quad (y\alpha,y\alpha,y\alpha,y\alpha) \eot{Q} R \quad \text{ and } \quad (z_2,x\alpha,z_2,y\alpha) \eot{Q} R,
	\]
	which implies that 
	\[
	(x\alpha,y\alpha,z_1) \eot{Q} R' \quad \text{ and } \quad (y\alpha,y\alpha,y\alpha) \eot{Q} R' \quad \text{ and } \quad (x\alpha,z_2,y\alpha) \eot{Q} R',
	\]
	and since $R'$ is majority selecting, it follows that $(x\alpha,y\alpha,y\alpha) \eot{Q} R'$. Thus, there exists $\phi:Q \rightarrow R_0'$ such that $r'\phi = (x,y,y)\alpha$. Now, take the pullback of $e$ along $\phi$, to obtain the diagram below: 
	\[
	\xymatrix{
		Q' \ar[r]^z \ar@{->>}[d]_{\alpha'} & R_0 \ar@{->>}[d]^{e} \ar[dr]^r \\ 
		Q \ar@{->>}[d]_{\alpha} \ar[r]^{\phi} & R_0' \ar[d]^{r'} & X^4\ar[dl]^{(\pi_2,\pi_3,\pi_4)} \\
		S \ar[r]_{(x,y,y)} & X^3
	}
	\]
	Then, if we let $p = \pi_1 r z$, it follows that $rz = (p, x\alpha\alpha', y\alpha\alpha', y\alpha\alpha')$. Now by construction of $R$, it follows that $(x\alpha\alpha', p) \eot{Q'} A$ and $(p,y\alpha\alpha') \eot{Q'} B \cap C$, so that
	\[
	(x\alpha\alpha', y\alpha\alpha')\eot{Q'} A \circ (B \cap C) \implies (x,y) \eot{S} A \circ (B \cap C).
	\]
	- by Remark~\ref{rem-reg-eot}. 
\end{proof}

\begin{lemma} \label{lem-effective-characterization-implies-majority}
	Let $\C$ be a regular category such that for any three effective equivalence relations $\alpha, \beta,\gamma$ on an object $X$, we have 
	\[
	\alpha \cap (\beta \circ \gamma) = (\alpha \cap \beta) \circ (\alpha \cap \gamma)
	\]
	then $\C$ is a majority category.
\end{lemma}

\begin{proof}
	Consider the ternary relation $R$ represented by the monomorphism $(r_1,r_2,r_3):R_0 \rightarrow X\times Y \times Z$, we will show that it is majority selecting in the sense of Definition~\ref{def-majority-selecting}. Let 
	\[
	x,x': S \rightarrow X, \quad y,y': S \rightarrow Y, \quad  z,z':S \rightarrow Z,
	\]
	and $a,b,c:S \rightarrow R_0$ be any morphisms in $\C$ such that the diagrams:
	\[
	\xymatrix{
		& R_0 \ar[d] & & R_0 \ar[d]  & & R_0 \ar[d]  \\
		S \ar[ru]^{a} \ar[r]_-{(x,y,z')} & X\times Y \times Z & S \ar[ru]^{b} \ar[r]_-{(x,y',z)} & X \times Y \times Z & S \ar[ru]^{c} \ar[r]_-{(x',y,z)} & X \times Y \times Z
	}
	\]
	commute. Consider the kernel congruences $\alpha, \beta, \gamma$ on $R$ formed from taking the kernel pairs of $r_1, r_2, r_3$ respectively. Then $(a,c) \eot{S} \beta \cap (\alpha \circ \gamma)\implies (a,c) \eot{S} (\beta \cap \alpha) \circ (\beta \cap \gamma)$, so that there exists a regular epimorphism $e:Q \rightarrow S$ and a morphism $b: Q \rightarrow R_0$ such that $(ae, b)\eot{Q} (\beta \cap \alpha)$ and $(b, c e) \eot{Q}(\beta \cap \gamma)$ by Proposition~\ref{prop-relation-composition}. This implies that $x e = r_1 b $ and $y e = r_2 b$ and $ze =  r_3 b$, and therefore $(x,y,z)e \eot{Q} R$ which implies that $(x,y,z) \eot{S} R$ by Remark~\ref{rem-reg-eot}. 
\end{proof}

\begin{theorem} \label{thm-pixley-generalization}
Let $\C$ be a regular category, then the following are equivalent for $\C$. 
\begin{enumerate}[(i)]
\item $\C$ is a majority category. 
\item For any three reflexive relations $A,B,C$ on any object $X$ in $\C$ we have:
\[
(A \circ B) \cap (A \circ C) \leqslant  A \circ (B \cap C).
\]
\item For any three reflexive relations $A,B,C$ on any object $X$ in $\C$ we have:
\[
A \cap (B\circ C) \leqslant  (A\cap B) \circ (A \cap C).
\]
\item For any equivalence relations  $\alpha,\beta,\gamma$ on any object $X$ in $\C$ we have:
\[
\alpha \cap (\beta\circ \gamma) =  (\alpha \cap \beta) \circ (\alpha \cap \gamma).
\]
\item For any effective equivalence relations  $\alpha,\beta,\gamma$ on any object $X$ in $\C$ we have:
\[
\alpha \cap (\beta\circ \gamma) =  (\alpha \cap \beta) \circ (\alpha \cap \gamma).
\]
\end{enumerate}
\end{theorem}

\begin{proof}
	Note that if $\C$ satisfies $(ii)$, then for any three reflexive relations $A,B,C$ on an object $X$ in $\C$ we have
	\[
	(B \circ A) \cap (C \circ A) \leqslant  (B \cap C) \circ A.
	\]
	This is because we may take the double opposite of the left-hand side: 
	\begin{align*}
	(B \circ A) \cap (C \circ A) &= (((B \circ A) \cap (C \circ A))^{\op})^{\op} \\
	&= ((B \circ A)^{\op} \cap (C \circ A)^{\op})^{\op} \\
	&= ((A \circ B) \cap (A \circ C))^{\op} \leqslant (A \circ (B \cap C))^{\op} = (B \cap C) \circ A
	\end{align*}
	Now, to prove the theorem above it suffices to show the implication $(ii) \implies (iii)$, since Lemma~\ref{lem-majority-implies-reflexive-relation-characterization} gives $(i) \implies (ii)$, and Lemma~\ref{lem-effective-characterization-implies-majority} gives $(v) \implies (i)$, and $(iii) \implies (iv)$ is trivial. Suppose that $(ii)$ holds, then we have:
	\begin{align*}
	(A \cap B) \circ (A \cap C) &\geqslant ((A \cap B) \circ A) \cap ((A \cap B)\circ C) \\
	&\geqslant ((A \cap B) \circ A) \cap (A\circ C) \cap (B \circ C) \\
	&\geqslant A \cap (B \circ C),
	\end{align*}
	by repeated application of $(iv)$.
\end{proof}
Given a morphism $f: X \rightarrow Y$ in a regular category $\C$ and a subobject $S \leqslant X \times X$, we will write $f(S)$ for the subobject $(f \times f)(S)$, and similarly we write $f^{-1}(S)$ for $(f\times f)^{-1}(S)$.  If $\C$ is regular, then we have: 
\[
f^{-1}f(S) = K \circ S \circ K,
\]
where $K$ is the kernel equivalence relation on $X$ associated to $f$. D. Bourn showed in \cite{Bou05}, that a regular Mal'tsev category is congruence distributive if and only if for any regular epimorphism $f:X \rightarrow Y$ and any equivalence relations $\alpha, \beta \in \Eq{X}$, we have $f(\alpha \cap \beta) = f(\alpha) \cap f(\beta)$ (in fact this was shown more generally for Goursat categories). The proof of the following proposition is essentially the proof of Theorem~2.1 in \cite{Bou05}, however we include it for completeness. 
\begin{proposition} \label{prop-preserve-meet}
	Let $\C$ be a regular category, then the following are equivalent. 
	\begin{enumerate}[(i)]
		\item For any regular epimorphism $f:X \rightarrow Y$, and any reflexive relations $R,S \in \Eq{X}$ we have $f(R \cap S) = f(R) \cap f(S)$. 
		\item For any three reflexive relations $R,S,T$ on any object $X$ in $\C$, we have 
		\[
		(T \circ R \circ T) \cap (T \circ S \circ T) = (T \circ R \cap S \circ T). 
		\]
	\end{enumerate}
\end{proposition}
The proof below is essentially that which can be found in \cite{Bou05}, however we include a sketch for completeness. 
\begin{proof}[Proof Sketch]
	For $(i) \implies (ii)$: suppose that $(r_1,r_2):R_0 \rightarrow X \times X$ and $(s_1,s_2):S_0 \rightarrow X \times X$  and $(t_1,t_2):T_0 \rightarrow X \times X$ represent $R,S,T$ respectively. Note that since $T$ is reflexive, that $t_1$ and $t_2$ are regular epimorphisms (as they are split epimorphisms). Also, we have $t_2(t_1^{-1}(R)) = T \circ R \circ T$ and $t_2(t_1^{-1}(S)) = T \circ S\circ T$, so that: 
	\begin{align*}
	(T \circ R \circ T) \cap (T \circ S \circ T) &= t_2(t_1^{-1}(R))  \cap t_2(t_1^{-1}(S)) \\
	&= t_2(t_1^{-1}(R)  \cap t_1^{-1}(RS))  \\
	&= t_2(t_1^{-1}(R \cap S)) \\
	& = T \circ (R \cap S) \circ T.
	\end{align*}
	For $(ii) \implies (i)$: suppose that $f:X \rightarrow Y$ is any regular epimorphism, then we have that 
	\begin{align*}
	f(R \cap S) &= f(f^{-1}f(R \cap S)) \\
	&= f(K \circ (R\cap S) \circ K) \\
	&= f((K \circ R \circ K)\cap (K \circ S \circ K)) \\
	&= f(f^{-1}f(R) \cap f^{-1}f(S)) \\
	& = ff^{-1}(f(R) \cap f(S)) \\
	& = f(R) \cap f(S).
	\end{align*}
\end{proof}
\begin{corollary}
	For any regular epimorphism $f:X \rightarrow Y$ in a regular majority category $\C$, and any reflexive relations $R, S$ on $X$ we have
	\[
	f(R \cap S) = f(R) \cap f(S).
	\]
\end{corollary}
\begin{proof}
	We show that any regular majority category $\C$ satisfies $(ii)$ of Proposition~\ref{prop-preserve-meet}: suppose that $R,S,T$ are reflexive relations on an object $X$ in $\C$. Then we always have:
	\[
	T \circ (R\cap S) \circ T \leqslant (T \circ R \circ T) \cap (T \circ S \circ T).
	\]
	For the reverse inequality, we have 
	\begin{align*}
	T \circ (R\cap S) \circ T &\geqslant ((T \circ R) \cap (T \circ S)) \circ T \\
	&\geqslant (T \circ R \circ T) \cap (T \circ S \circ T),
	\end{align*}
	 - by (ii) of Theorem~\ref{thm-pixley-generalization}. 
\end{proof}
Recall that if $\C$ is a regular Mal'tsev category, then for any two equivalence relations $\alpha, \beta$ on any object $X$ in $\C$, the join $\alpha \vee \beta$ exists, and is given by  
\[
\alpha \circ \beta = \alpha \vee \gamma.
\]
\begin{corollary}
	Let $\C$ be a regular Mal'tsev category. Then $\C$ is congruence distributive if and only if $\C$ is a majority category.
\end{corollary}
The notion of a \textit{protoarithmetical} category, which was first introduced by D.~Bourn in \cite{Bou02}, has a strong relation to majority categories: every finitely complete Mal'tsev majority category is protoarithmetical (see \cite{Hoe18a}), and a Barr exact category $\C$ is protoarithmetical if and only if it is both Mal'tsev and a majority category. In the regular context, we have the following characterization of protoarithmetical categories in terms of a certain \emph{weak congruence distributivity}. 
\begin{theorem*}[\cite{Bou01}]
	For a regular Mal'tsev category $\C$ the following are equivalent. 
	\begin{enumerate}
		\item $\C$ is a protoarithmetical category.
		\item For any three equivalence relations $\alpha, \beta, \gamma$ on any object $X$ in $\C$, if $\alpha \cap\beta = \Delta_X$ and $\alpha \cap \gamma = \Delta_X$ then $\alpha \cap(\beta \vee \gamma) = \Delta_X$. 
	\end{enumerate}
\end{theorem*}
For example, in \cite{Bou05}, the author remarks that the categories $\mathrm{VonReg}(\Top)$ of topological Von Neumann regular rings, or $\mathrm{BoRg}(\Top)$ of topological Boolean rings, are both regular protoarithmetical categories, and therefore they satisfy the above weak version of congruence distributivity. It is then remarked that 'it is far less clear, at this point, if they are fully congruence
distributive or not'. However, the corollary above shows that it is indeed the case that they are fully congruence distributive (since they are majority categories). The general question of whether weak congruence distributivity is equivalent to full congruence distributivity, is given in the negative in \cite{Hoe18a}. There the author constructs a regular protoarithmetical category which is not a majority category, and hence which is not congruence distributive. 

\section{Bergman's Double-projection Theorem for regular categories} 
If $S$ is any sublattice of a finite product of lattices $L_1 \times \cdots \times L_n$,  then $S$ is uniquely determined by its images under the canonical projections $\pi_{i,j}:L_1 \times \cdots \times L_n \rightarrow L_i \times L_j$. As mentioned in the introduction, this is property of $\Lat$ extends to a characterization of varieties which admit a majority term. 
\begin{theorem*}[K.A. Baker and A.F. Pixley \cite{BP75}]
	The following are equivalent for a variety $\V$ of algebras. 
	\begin{enumerate}
		\item $\V$ admits a majority term. 
		\item Any subalgebra $S$ of a finite product $\prod\limits_{i = 1}^n A_i$ of algebras in $\V$ is uniquely determined by its two-fold images under the canonical projections $\pi_{i,j} :\prod\limits_{i = 1}^n A_i \rightarrow A_i \times A_j$. 
	\end{enumerate}
\end{theorem*}
The aim of this section is to generalize this theorem to a characterization of regular majority categories. 
\begin{definition}
	Let $\C$ be a regular category and let $I$ be a set, and $J \subseteq I$ any subset. Suppose that $(A_i)_{i \in I}$ is a family of objects in $\C$, such that both products $\prod\limits_{i \in I} A_i$ and  $\prod\limits_{j \in J} A_j$ exist. Then for any subobject $S$ of $\prod\limits_{i \in I} A_i$, the image of $S$ under the canonical morphism: 
	\[
	\prod\limits_{i \in I} A_i \xrightarrow{\pi_{J}} \prod\limits_{j \in J} A_j, 
	\]
	is called the $J$-image of $S$ in $\prod\limits_{j \in J} A_j$ and is denoted by $S_J$.
\end{definition}
\begin{definition}
	Let $\C$ be a regular category and let $I$ be a set and $\mathcal{J} = (I_j)_{j \in J}$ a family of subsets of $I$. Then the product $\prod\limits_{i \in I} A_i$ (if it exists) is said to have $\mathcal{J}$-fold subobject decompositions if it satisfies the following property: for any two subobjects $S,T$ of $\prod\limits_{i \in I} A_i$, if $S_{I_j} = T_{I_j}$ for any $j \in J$, then $S = T$. In other words, we say that every subobject of  $\prod\limits_{i \in I} A_i$ is uniquely determined by its $\mathcal{J}$-fold images.  If every product  indexed by $I$ (which exists) has $\mathcal{J}$-fold subobject decompositions, then we say that $\C$ has $\mathcal{J}$-fold subobject decompositions. 
\end{definition}
\begin{proposition} \label{prop-J-fold-subobject-decompositions}
	Let $\C$ be a regular category, and let $I$ be a set and $\mathcal{J} = (I_j)_{j \in J}$ a family of subsets of $I$. The following are equivalent for a family $(A_i)_{i \in I}$  of objects in $\C$.
	\begin{enumerate}[(i)]
		\item $\prod\limits_{i \in I} A_i$ has $\mathcal{J}$-fold subobject decompositions.
		\item For any monomorphism $s: S \rightarrow \prod\limits_{i \in I} A_i$, the diagram
		\[
		\xymatrix{
			S\ar[d]_s \ar[rrr]^-{(e_{I_j})_{j \in I}} & & &\prod\limits_{j\in J} S_{I_j} \ar[d]^-{\prod\limits_{j \in J}s_{I_j}} \\
			\prod\limits_{i\in I} A_i \ar[rrr]_-{(\pi_{I_j})_{j \in J}} & & & \prod\limits_{j\in J} (\prod\limits_{k \in I_j} A_k)
		}
		\]
		is a pullback, where
		\[
		S \xrightarrow{e_{I_j}} S_{I_j} \xrightarrow{s_{_{I_j}}} \prod\limits_{k \in I_j} A_k
		\]
		is a regular epi, mono factorization of $\pi_{I_j} s$. 
	\end{enumerate}
\end{proposition}
\begin{proof}
	For (i) $\implies$ (ii):  let $s: S \rightarrow \prod\limits_{i \in I} A_i$ be any monomorphism, and consider the diagram below where the square is a pullback: 
	\[
	\xymatrix{
		S\ar@/_1pc/[ddr]_s \ar@/^1pc/[rrrrd]^-{(e_{I_j})_{j \in I}} \ar@{..>}[rd]^\phi	& & & & \\
		&	P\ar[rrr]^{\alpha} \ar[d]^p& & & \prod\limits_{j\in J} S_{I_j} \ar[d]^-{\prod\limits_{j \in J}s_{I_j} = \beta} & \\
		&	\prod\limits_{i\in I} A_i \ar[rrr]_-{(\pi_{I_j})_{j \in J}} & & & \prod\limits_{j\in J} (\prod\limits_{k \in I_j} A_k) 
	}
	\]
	By construction, the outer rectangle commutes, so that the dotted arrow $\phi$ exists. We claim that the subobject represented by $p$ has the same two fold images as the subobject represented by $s$. Let $j \in J$ be any element then since $(\pi_j\alpha) \phi = e_{I_j}$ is a regular epimorphism,  $\pi_j \alpha$ is also a regular epimorphism. Then the factorization  $s_{I_j} (\pi_j\alpha) = \pi_{I_j} p$ is an image factorization, therefore the $I_j$-image of the subobject represented by  $p$ in $\prod\limits_{i \in I} A_i$ is $S_{I_j}$. Therefore, the subobjects represented by $s$ and $p$ are the same, so that $\phi$ is an isomorphism, which implies that the outer rectangle is a pullback.  Finally, (ii) $\implies$ (i) follows from the universal property of pullback.
\end{proof}
\begin{definition}  \label{def-k-fold-subobject-decompositions}
	Let $I$ be any set. For any regular category $\C$, we say that $\C$ has $k$-fold subobject decompositions over $I$ (where $k$ is a positive integer) if it has $\mathcal{J}$-fold decompositions, where $\mathcal{J}$ is the set of all subsets of $I$ of size $k$.  If $I$ is a countable set, then we say that $\C$ has countable $k$-fold subobject decompositions. If $\C$ has $k$-fold subobject decompositions over any finite set, then $\C$ is said to have finite $k$-fold subobject decompositions.   
\end{definition}
\begin{theorem} \label{thm-relationDecompositionImpiesMajority}
	Let $\C$ be a regular category. If $\C$ has finite 2-fold subobject decompositions, then $\C$ is a majority category.
\end{theorem}
\begin{proof}
Suppose that $R$ is any subobject of $A \times B \times C$ represented by $(r_1,r_2,r_3): R_0 \rightarrow A \times B \times C$. Let $r_{1,2}: R_{1,2} \rightarrow A \times B$ and $r_{1,3}:R_{1,3} \rightarrow A \times C$ and $r_{2,3}:R_{2,3} \rightarrow B \times C$ be the monomorphisms formed from taking the mono part of the image-factorization of $(r_1,r_2)$, $(r_1,r_3)$ and $(r_2,r_3)$ respectively. Consider the pullback square below:
	\[
	\xymatrix{ 
		P_0 \ar[rrrr] \ar[d]_{p = (p_1,p_2,p_3)} & & & & R_{1,2} \times R_{1,3} \times R_{2,3} \ar[d]^{r_{1,2}\times r_{1,3}\times r_{2,3}} \\
		A \times B\times C \ar[rrrr]_-{((\pi_1, \pi_2),(\pi_1, \pi_3),(\pi_2, \pi_3))} & &  & & (A\times B) \times (A \times C) \times (B \times C)
	}
	\]
	It is easily seen that $P$ is majority selecting in the sense of Definition~\ref{def-majority-selecting}, and therefore by Proposition~\ref{prop-J-fold-subobject-decompositions} we have $P = R$, so that $R$ is majority selecting.  
\end{proof}
Given any relation $R$ on a product $X \times Y$, we can consider the image of $R$ under the canonical projections $(X\times Y)^2 \rightarrow X^2$ and $(X\times Y)^2 \rightarrow Y^2$ which give two relations $R_1$ and $R_2$ on $X$ and $Y$, respectively. Conversely, given $R_1$ and $R_2$ represented by $r_1:R_0 \rightarrow X \times X$ and $r_2:R_0' \rightarrow Y \times Y$ respectively, then the composite morphism
\[
R_0 \times R_0' \xrightarrow{r_1 \times r_2} (X \times X) \times (Y \times Y) \xrightarrow{\phi} (X \times Y)^2
\]
is a mono (where $\phi$ is the canonical 'transpose' isomorphism), which represents a relation $R_1 \times_T R_2$ on $(X \times Y)$. Note, that we always have $R \leqslant R_1 \times_T R_2$. 
\begin{definition}
	A  regular category $\C$ is said to have \emph{directly decomposable reflexive relations}, if for any reflexive relation $R$ on a product $X \times Y$ in $\C$, we have $R_1 \times_T R_2 = R$.
\end{definition}
\begin{example}
	The category $\mathbf{Ring}$ of unitary rings has directly decomposable reflexive relations, and the category $\Grp$ does not. For a proof of this, we refer the reader to Example~3.9 in \cite{HoePhd}. 
\end{example}
\begin{proposition}
	Let $\C$ be any regular category with which has finite two-fold subobject decompositions. Then $\C$ has directly decomposable reflexive relations. 
\end{proposition}
\begin{proof}
	For any relation $R$ on a product $X \times Y$, its easy to see that both $R$ and $R_1 \times_T R_2$ have the same two-fold projections, when viewed as subobjects of $X \times Y \times X \times Y$.  
\end{proof}
In  Theorem~\ref{thm-main-characterization}, we will see that any regular majority category has finite 2-fold subobject decompositions. This is then the categorical analogue of the lattice-theoretic double-projection theorem of Bergman mentioned in the introduction. As we will see in the next section, there are no finitary varieties which have countable 2-subobject decompositions, however, there are infinitary varieties which do. 
\subsection{Infinite subobject decompositions}

\begin{proposition} 
	The only finitary varieties of algebras $\V$ which have countable $2$-fold subobject decompositions are trivial, i.e., each algebra in $\V$ has at most $1$ element. 
\end{proposition}
\begin{proof}
	Suppose that $\V$ is a finitary variety which has countable $2$-fold subobject decompositions. Consider the set theoretic maps:
	\[
	f_n: \N \rightarrow \ord{n},  \quad x \longmapsto \begin{cases}
	x & x \leqslant n \\
	n & x > n
	\end{cases}
	\]
	Let $F = F_{\V}(\N)$ and $F_n = F_{\V}(\ord{n})$, then each $f_n$ induces a homomorphism $\overline{f_n}: F \rightarrow F_n$ via the free algebra in $\V$. Now let $f$ be the induced homomorphism into the product of the $F_i's$
	\[
	\xymatrix{
		F \ar@{..>}[rr]^f \ar[rrd]_{\overline{f_n}} &  &\prod\limits_{n \in \N} F_n \ar[d]^{\pi_n}\\
		&& F_n  
	}
	\]
	Then $f$ is a monomorphism. Let $F_{i,j}$ be the two-fold image of $F$ in $F_i \times F_j$, and consider the image factorization:
	\[
	F \xrightarrow{e_{i,j}} F_{i,j} \xrightarrow{f_{i,j}} F_i \times F_j.
	\]
	where  $f_{i,j}$ is the canonical inclusion, and $e_{i,j}$ the canonical projection. 
	Let $g_n:F_1 \rightarrow F_n$ be the homomorphism sending $1$ to $n$, and let $g = (g_n)_{n \in N}$. Now, for any $i,j \in \N$ we have that $(i,j) \in F_{i,j}$ since if $i \leqslant j$ then $f_i(j) = i$ and $f_j(j) =j$. Consider the homomorphism $g_{i,j}:F_1 \rightarrow F_{i,j}$ sending $1$ to $(i,j)$, this gives the following commutative diagram:
	\[
	\xymatrix{
		F_1 \ar[d]_g \ar[rrr]^{(g_{i,j})_{i,j \in \N}} & & &\prod\limits_{i,j \in \N} F_{i,j} \ar[d]^-{\prod\limits_{i,j \in \N} f_{i,j}} \\
		\prod\limits_{n\in \N} F_n \ar[rrr]_-{(\pi_i,\pi_j)_{i,j \in \N}} & & & \prod\limits_{i,j \in \N} F_i \times F_j
	}
	\]
	Now, by Proposition~\ref{prop-J-fold-subobject-decompositions},  the square:
	\[
	\xymatrix{
		F \ar[d]_f \ar[rrr]^{(e_{i,j})_{i,j \in \N}} & & &\prod\limits_{i,j \in \N} F_{i,j} \ar[d]^-{\prod\limits_{i,j \in \N} f_{i,j}} \\
		\prod\limits_{n\in \N} F_n \ar[rrr]_-{(\pi_i,\pi_j)_{i,j \in \N}} & & & \prod\limits_{i,j \in \N} F_i \times F_j
	}
	\]
	is a pullback. Therefore, there exists a morphism $F_1 \rightarrow F$, making the relevant triangle commute. This amounts to the existence of an element $x \in F$ such that $\overline{f_n}(x) = n$ for any $n \in \N$. Since $x$ is an element of $F$ it follows that $x = t(a_1,a_2,....,a_k)$ where $t$ is a $k$-ary term, and $a_1,a_2,...,a_k \in \N$. Now, let $m = \max \{a_1,a_2,...,a_k\}$, then it follows that 
	\[
	m = \overline{f_m}(t(a_1,a_2,....,a_k) = t(f_m(a_1),f_m(a_2),....,f_m(a_k)) =  t(a_1,a_2,....,a_k),
	\]
	but also we have
	\[
	m + 1 = \overline{f_{m+1}}(t(a_1,a_2,....,a_k) = t(f_{m+1}(a_1),f_{m+1}(a_2),....,f_{m+1}(a_k)) =  t(a_1,a_2,....,a_k),
	\]
	so that in $F_{m+1} \models m = m + 1$.  This implies that every algebra in $\V$ has at most one element.
\end{proof}

In the above proof it is crucial that $\V$ be finitary, as the finiteness of $t$ allows us to select the maximum of $a_1,a_2,...,a_k$.  In what follows, we will see that there can be infinitary varieties with 2-fold subobject decompositions over any set $I$.   

Recall that if $I$ is an arbitrary set, then an $I$-complete lattice $L$ is one in which any family $(x_i)_{i \in I}$ has a meet and a join. A homomorphism $f: L \rightarrow M$ of $I$-complete lattices is a function which preserves joins and meets of families indexed by $I$. ln what follows we shall denote the category of $I$-complete lattices by $\Lat_{I}$. 
\begin{proposition}
	The category $\Lat_{I}$ of $I$-complete lattices has 2-fold subobject decompositions over $I$. 
\end{proposition}
\begin{proof}
	Suppose that $S \subseteq \prod\limits_{i \in I} L_i  = L$ is any $I$-complete sublattice of a product of $I$-complete lattices, and suppose that $\pi_k: \prod\limits_{i \in I} L_i  \rightarrow L_k$ are the canonical product projections. Then to show that $S$ satisfies (ii) of Proposition~\ref{prop-J-fold-subobject-decompositions} where $\mathcal{J}$ is the set of all $2$-element subsets of $I$,  amounts to showing that $S$ has the following property: for any $x \in L$, if for any $i, j \in I$ there exists $s \in S$ such that $\pi_i(x) =\pi_i(s)$ and $\pi_j(x) = \pi_j(s)$ -- (*),  then $x \in S$.  To that end, suppose that $x \in L$ satisfies (*), and let $s_{i,j} \in S$ be elements of $S$ with $\pi_i(s_{i,j}) = x_i$ and $\pi_j(s_{i,j}) = x_j$. Define the elements $s_j$ of $S$ as follows:
	\[
	s_j = \bigwedge \limits_{i \in I} s_{i,j}
	\]
	then  for any $i,j \in I$ we have $\pi_i(s_j) \leqslant \pi_i(s_i) = x_i$, since
	\[
	s_j = \bigwedge \limits_{i \in I} s_{i,j} \leqslant s_{i,j}\implies \pi_i(s_j) \leqslant \pi_i(s_{i,j}) = x_i = \pi_i(s_i)
	\]
	This implies that
	\[
	x = \bigvee\limits_{i \in I} s_i,
	\]
	so that $x \in S$.  
\end{proof}
\begin{proposition}
	If $I,J$ are infinite sets and $|I| < |J|$, then $\Lat_{I}$ does not have 2-fold subobject decompositions of size $J$. 
\end{proposition}
\begin{proof}
	Consider the subset $S$ of $\prod\limits_{j \in J} \mathbf{2}$ consisting of all elements $s \in \prod\limits_{j \in J} \mathbf{2}$ such that 
	\[
	|\{j \in J\mid \pi_j(s) = 1\}| \leqslant |I| 
	\]
	suppose that $(s_i)_{i \in I}$ is a collection of elements of $S$ and let $s = \bigvee\limits_{i \in I} s_i$.  Then it is easy to see that
	\[
	\{j \in J\mid \pi_j(s) = 1\} = \bigcup\limits_{i \in I} \{j \in J\mid \pi_j(s_i) = 1\}
	\]
	But then since $I$ is infinite, it follows that $|I \times I| = |I|$. Therefore we have:
	\[
	|\{j \in J\mid \pi_j(s) = 1\}| = |\bigcup\limits_{i \in I} \{j \in J\mid \pi_j(s_i) = 1\}| \leqslant |\bigsqcup\limits_{i \in I} I| = |I \times I| = |I|
	\]
	so that $s \in S$. Thus, $S$ is a sublattice of $\prod\limits_{j \in J} \mathbf{2}$. Moreover, $S$ has the same 2-fold projections as $\prod\limits_{j \in J} \mathbf{2}$, but is not equal to $\prod\limits_{j \in J} \mathbf{2}$.  As, for example, the top element of $ \prod\limits_{j \in J} \mathbf{2}$ is not contained in $S$. 
\end{proof}
Interestingly, many duals of geometric categories such as $\Ord^{\op}$, $\Grph^{\op}, G-\Set^{\op}$ have 2-fold subobject decompositions over any set. In general, given a coregular category $\C$, to show that $\C^{\op}$ has two-fold subobject decompositions over $I$, we have to show that if $r:\bigsqcup\limits_{i \in I} X_i \rightarrow R$ and $t:\bigsqcup\limits_{i \in I} X_i \rightarrow T$ are any two epimorphisms in $\C$ with the same two-fold coimages, then $r$ is isomorphic to $t$ in the slice category $(\bigsqcup \limits_{i \in I} X_i \downarrow \C)$. This amounts to showing that if for any $i,j \in I$ we have the following commutative diagram of solid arrows
\begin{figure}[h]
	\[
	\xymatrix{
		& R_{i,j} \ar[dd]^(.35){\phi_{i,j}}|\hole \ar[rr]^{r_{i,j}}& & R \ar@{..>}[dd]^(.35){\phi}\\ 
		X_i \sqcup X_j \ar[rr]\ar[ur]^{\alpha_{i,j}} \ar[dr]_{\beta_{i,j}} & & \bigsqcup\limits_{i \in I} X_i \ar[ur]_r \ar[dr]^t\\
		& T_{i,j} \ar[rr]_{t_{i,j}} & & T
	}
	\]
\end{figure}
where $\phi_{i,j}$ are isomorphisms, $r_{i,j}\alpha_{i,j}$ and  $t_{i,j}\beta_{i,j}$ are the canonical coimage factorizations, then the dotted arrow $\phi$ exists, is an isomorphism, and makes the diagram above commute.  For $\C = \Set$, we define the $\phi$ as follows: if $x \in R$, then select an element $y \in r^{-1}(x)$ and set $\phi(x) = t(y)$. To see that this is well-defined, suppose that $y,y' \in r^{-1}(x)$ then there exists $i,j \in I$ such that $y\in X_i$ and $y' \in X_j$. Now, $\alpha_{i,j}(y) = \alpha_{i,j}(y')$ since $r_{i,j}$ is a monomorphism, and therefore $\phi_{i,j}\alpha_{i,j}(y) = \beta_{i,j}(y) = \beta_{i,j}(y')  = \phi_{i,j}\alpha_{i,j}(y')$ so that $t_{i,j}\beta_{i,j}(y) = t_{i,j}\beta_{i,j}(y')$ which implies $t(y) = t(y')$.  In each of the coregular categories $ G-\Set, \Ord, \Grph$ it is easy to see that the map defined above, is actually and isomorphism in each category. This shows, in particular, that the categories mentioned above are majority categories by Theorem~\ref{thm-relationDecompositionImpiesMajority}. 

 Perhaps it is surprising that the category $\Top^{\op}$ does not have 2-fold subobject decompositions over arbitrary sets. Indeed, it does not even have countable two-fold subobject decompositions: consider $\Q$ together with the subspace topology induced by $\mathbb{R}$. Define the continuous maps $f_a:\Q\rightarrow \Q^2$ by $f_a(x) = (a,x)$. The induced continuous map $f$ in the diagram
 \[
 \xymatrix{
 	\bigsqcup\limits_{a \in \Q} \Q \ar@{..>}[rr]^f & & \Q^2 \\
 	\Q \ar[u]^-{\iota_a} \ar[urr]_{f_a}
 }
 \]
 is an epimorphism. Moreover, $f$ has the same two-fold co-images as the identity on $\bigsqcup\limits_{a \in \Q} \Q$, so that if $\Top^{\op}$ had two-fold subobject decompositions, then we would have $\bigsqcup\limits_{a \in \Q} \Q \simeq \Q \times \Q$ --- which is a contradiction. 
\begin{proposition}
	$\Top^{\op}$ has finite two-fold subobject decompositions. 
\end{proposition}
Recall that regular monomorphisms in $\Top$ are precisely the embeddings of spaces. 
\begin{proof}
	We will show that in the above figure, that $\phi$ is a homeomorphism, provided that $I$ is finite. We first show that $\phi$ preserves open sets:  let $U \subseteq R$ be any open set in $R$, then for any $i,j$ we have that $R_{i,j} \cap U$ is open in $R_{i,j}$, which implies that $\phi_{i,j}(U) $ is open in $T_{i,j}$ since each $\phi_{i,j}$ is a homeomorphism. Therefore, there exists an open set $V_{i,j} \subseteq T$ such that $\phi_{i,j}(R_{i,j} \cap U)  = T_{i,j} \cap V_{i,j}$, and hence we have: 
	\[
	T_{i,j} \cap \phi(U) = \phi(R_{i,j} \cap U) = \phi_{i,j}(R_{i,j} \cap U)  = T_{i,j} \cap V_{i,j}.
	\]
	Let $V_i = \bigcap\limits_{j \in I} V_{i,j}$ and $T_i = T_{i,i}$. Then we have 
	\[
	\bigcap\limits_{j \in J} (T_{i,j} \cap \phi(U)) = \bigcap\limits_{j \in J} (T_{i,j} \cap V_{i,j} ) \implies T_i \cap \phi(U) = T_i \cap V_i
	\]
	then we will show that $\bigcup\limits_{i \in I} V_i = \phi(U)$. For the direction $\phi(U) \subseteq \bigcup\limits_{i \in I} V_i$: let $x \in \phi(U)$ then there exists $j \in I$ such that $x \in T_{j}$, so that
	\[
	x \in \phi(U) \cap T_{j} \implies  x \in T_j\cap V_j \implies x \in \bigcup\limits_{i \in I} V_i 
	\]
	For the reverse inclusion $\bigcup\limits_{i \in I} V_i \subseteq \phi(U)$: suppose that $x \in V_i$ for some $i \in I$. Then there exists $j \in I$ such that $x \in T_{i,j}$ and therefore,
	\[
	x \in T_{i,j} \cap V_i \implies x \in T_{i,j} \cap V_{i,j} \implies x \in T_{i,j} \cap\phi(U) \implies x \in \phi(U).
	\]
\end{proof}
\section{The Pairwise Chinese Remainder Theorem in a category}

Let $\C$ be a regular category and $X$ an object of $\C$. If $\theta$ is an equivalence relation on $X$ and $a,b:S \rightarrow X$ morphisms in $\C$, then we will write $a \equiv b \mod \theta$ if $(a,b) \eot{S} \theta$ in what follows. Given an object $X$ of a category $\C$, morphisms $a_1,a_2,...,a_m:S \rightarrow X$ and equivalence relations $\theta_1,\theta_2,...,\theta_m$, we will be concerned with solving the system of congruence equations:
\begin{align*}
x \equiv& a_1 \mod \theta_1, \\
x \equiv& a_2 \mod \theta_2, \\ 
& \vdots  \tag{$*$}\\
x \equiv& a_m \mod \theta_m.
\end{align*}
\begin{definition}
	An approximate solution to the system above consists of a morphism $a:Q \rightarrow X$ (the approximate solution), together with a regular epimorphism $\alpha: Q \rightarrow S$ (the approximation of $a$), such that for any $i \in \{1,2,...,m\}$ we have
	\[
	a \equiv a_i \alpha \mod \theta_i.
	\]
	If such an $a:Q \rightarrow X$ and $\alpha:Q \rightarrow S$ exist, then the above system ($*$) is said to be approximately solvable. The above system ($*$) is said to be approximately pairwise solvable, if for any $i,j \in \{1,2,...,m\}$ the system 
	\begin{align*}
	x\equiv& a_i \mod \theta_i, \\
	x \equiv& a_j \mod \theta_j 
	\end{align*}
	is approximately solvable.
\end{definition}
\begin{remark}
	The above notion is similar to the notion of an approximate operation in the sense of \cite{BJ08}, in how it compares with the ordinary notion of solution to a system of equations. 
\end{remark}
\begin{definition}[PCRT] \label{def-PCRT}
	Let $X$ be an object of a regular category $\C$, then $X$ is said to satisfy the Pairwise Chinese Remainder Theorem, if for any morphisms $a_1,a_2,...,a_m:S \rightarrow X$, and any effective equivalence relations $\theta_1,\theta_2,...,\theta_m$, if the system 
	\begin{align*}
	x \equiv& a_1 \mod \theta_1, \\
	x \equiv& a_2 \mod \theta_2, \\ 
	& \vdots \\
	x \equiv& a_m \mod \theta_m
	\end{align*}
	is approximately pairwise solvable, then it is approximately solvable. If every object of $\C$ satisfies the PCRT, then we say that $\C$ satisfies the PCRT, or that the PCRT holds in $\C$.
\end{definition}
\begin{lemma} \label{lem-regular-category}
	If $\alpha_i': Q_i \rightarrow A$ and $\beta_i':Q_i \rightarrow B$ are regular epimorphisms making the diagram
	\[
	\xymatrix{
		Q_i \ar[r]^{\beta_i'} \ar[d]_{\alpha_i'} & B\ar[d]^{b_i} \\
		A \ar[r]_{a_i} & C_i 
	}
	\]
	commute, then there exist regular epimorphisms $\alpha:Q \rightarrow A$ and $\beta_i:Q \rightarrow B$ making the diagram 
	\[
	\xymatrix{
		Q \ar[r]^{\beta_i} \ar[d]_{\alpha} & B\ar[d]^{b_i} \\
		A \ar[r]_{a_i} & C_i 
	}
	\]
	commute for any $i \in \{1,2,...,n\}$.
\end{lemma}
\begin{proof}
	Simply consider the limit of the diagram:
	\[
	Q_i \xrightarrow{\alpha_i'} A
	\]
	where $i$ ranges from $1$ to $n$. This produces a family of regular epimorphisms $p_i:Q \rightarrow Q_i$ making the diagram
	\[
	\xymatrix{
		Q \ar[d]_{p_i} \ar[dr]^{\alpha}\\
		Q_i \ar[r] _{\alpha'_i} & A 
	}
	\]
	commute, where $\alpha$ is any composite $\alpha_i'p_i$ where $i \in \{1,2,...,n\}$. Then defining $\beta_i = \beta'_i p_i$, it follows that $\alpha$ and $\beta_i$ satisfy the required properties. 
\end{proof}
\begin{lemma} \label{lem-PCRT-implies-subobject-decompositions}
	Let $\C$ be a regular category, then (i) $\implies$ (ii) where
	\begin{enumerate}[(i)]
		\item The PCRT holds in $\C$. 
		\item $\C$ has finite 2-fold subobject decompositions. 
	\end{enumerate}
\end{lemma}
\begin{proof}
	Suppose that $C_1,C_2,....,C_r$ are any objects in $\C$, and let  $A,B$ be any subobjects of $C = C_1 \times C_2 \times \cdots \times C_r$, with representatives $a:A_0 \rightarrow C$ and $b:B_0 \rightarrow C$, and which have the same 2-fold images in $C_i \times C_j$. Let $\pi_i(a) = a_i$ and $\pi_i(b) = b_i$, then we will show that $A \leqslant B$. Consider the regular-epi mono factorizations of the morphisms $(a_i,a_j)$ and $(b_i,b_j)$ below: 
	\begin{align*}
	A_0 \xrightarrow{\alpha_{i,j}'} &R \xrightarrow{(r_i,r_j)} C_i \times C_j, \\ 
	B_0 \xrightarrow{\beta_{i,j}'} &T \xrightarrow{(t_i,t_j)} C_i \times C_j .
	\end{align*}
	Since $A$ and $B$ have the same two-fold images, there exists an isomorphism $\phi:R \rightarrow T$ such that $(t_i,t_j)\phi = (r_i,r_j)$. Now, we can pullback $\phi\alpha'_{i,j}$ along $\beta'_{i,j} $, and get two regular epimorphisms $\alpha''_{i,j}:Q_{i,j} \rightarrow A_0$ and $\beta''_{i,j}:Q_{i,j} \rightarrow B_0$ making the diagram 
	\[
	\xymatrix{
		Q_{i,j} \ar@{->>}[r]^{\beta''_{i,j}} \ar@{->>}[d]_{\alpha''_{i,j}} & B_0\ar[d]^-{(b_i, b_j)} \\
		A_0 \ar[r]_-{(a_i, a_j)} & C_i \times C_j
	}
	\]
	commute. Then by Lemma~\ref{lem-regular-category}, there exist regular epimorphisms $\alpha:Q \rightarrow  A_0$ and $\beta_{i,j}:Q \rightarrow B_0$ such that the diagram
	\[
	\xymatrix{
		Q \ar@{->>}[r]^{\beta_{i,j}} \ar@{->>}[d]_{\alpha} & B_0\ar[d]^-{(b_i, b_j)} \\
		A_0 \ar[r]_-{(a_i, a_j)} & C_i \times C_j
	}
	\]
	commutes for any $i,j \in \{1,2,...,r\}$. Now define $\beta_i = \beta_{i,i}$, and let $\theta_i$ be the kernel equivalence relation on $B_0$ defined by $b_i$. Then we have that
	\[
	\beta_{i,j}\equiv \beta_i \mod \theta_i \quad \text{and} \quad \beta_{i,j}\equiv \beta_j \mod \theta_j,
	\] 
	so that the system
	\[
	x \equiv \beta_i \mod \theta_i \tag{for $i = 1,2,...,r.$}
	\]
	is pairwise approximately solvable (the approximation in each  case is the identity on $Q$). 
	Therefore, by (i) there exists a regular epimorphism $\alpha':Q' \rightarrow Q$ and a morphism $\beta:Q' \rightarrow B_0$ such that
	\[
	\beta \equiv \beta_i \alpha' \mod \theta_i \tag{for $i = 1,2,...,r.$}
	\]
	This implies that 
	\begin{align*}
	b_i\beta_i\alpha' = b_i\beta \implies a_i\alpha \alpha' = \pi_i b\beta \implies \pi_i (a\alpha \alpha') = \pi_i (b\beta),
	\end{align*}
	for any $i \in \{1,2,...,r\}$, and therefore, $a\alpha \alpha' = b\beta$. Therefore the diagram of solid arrows 
	\[
	\xymatrix{
		Q'\ar@{->>}[d]_{\alpha\alpha'} \ar[r]^{\beta} & B_0 \ar[d]_b\\ 
		A_0 \ar[r]_{a} \ar@{..>}[ur] & C
	}
	\]
	commutes, and the dotted arrow exists since $\alpha\alpha'$ is a regular epimorphism and $b$ is a monomorphism. This shows that $A \leqslant B$, and similarly we may get $B \leqslant A$. 
\end{proof}

\begin{lemma} \label{lem-majority-selecting}
	Let $\C$ be a majority category with finite products, and $R \leqslant A_1 \times A_2 \times \cdots \times A_n$ any $n$-ary relation with $n \geqslant 3$. If 
	\[
	(x,a_2,a_3,...,a_n) \eot{S} R \quad \text{and} \quad (a_1,y,a_2,a_3,...,a_n) \eot{S} R \quad  \text{and} \quad (a_1,a_2,z,...,a_n) \eot{S} R,
	\]
	then 
	\[
	(a_1,a_2,a_3,...,a_n) \eot{S} R.
	\]
\end{lemma}
\begin{proof}
	Follows by Definition~\ref{def-majority-selecting} from the fact that $R$ is a ternary relation between $A_1,A_2$ and $A_3 \times \cdots \times A_n$, which must be majority-selecting.
\end{proof}

\begin{lemma} \label{lem-majorityImpliesPCRT}
	If $\C$ is a regular majority category, then the Pairwise Chinese Remainder Theorem holds for $\C$.
\end{lemma}
\noindent
Consider the system of congruences from Definition~\ref{def-PCRT}:
\begin{align*}
x \equiv& a_1 \mod \theta_1, \\
x \equiv& a_2 \mod \theta_2, \\ 
& \vdots  \tag{$*$}\\
x \equiv& a_m \mod \theta_m.
\end{align*}
In the proof below, we will show that in any regular majority category $\C$, if any system of congruences of length $m$ is approximately solvable as soon as it is pairwise approximately solvable, then any system of length $m+1$ is approximately solvable as soon as it is pairwise approximately solvable. The result will then follow by induction, since in any regular category $\C$, any system of length $2$ is approximately solvable if and only if it is pairwise approximately solvable. 
\begin{proof} 
	Suppose that $m > 2$ is any natural number, and suppose that any system of congruences in $\C$ of length $m$ is approximately solvable as soon as it is pairwise approximately solvable. Let $X$ be any object in $\C$,  $a_1,a_2,...,a_{m+1}: S \rightarrow X$ any morphisms, and $\theta_1,\theta_2, ..., \theta_{m+1}$ any effective equivalence relations of the morphisms $f_1:X \rightarrow X_1, f_2:X \rightarrow X_2 ,..., f_{m+1}:X \rightarrow X_{m+1}$ respectively. Suppose that the system 
	\begin{align*}
	x \equiv& a_1 \mod \theta_1, \\
	x \equiv& a_2 \mod \theta_2,  \tag{$*$}\\
	&\vdots \\
	x \equiv& a_{m+1} \mod \theta_{m+1},
	\end{align*}
	is pairwise approximately solvable. By assumption, the three systems obtained from removing the first, second and third rows from ($*$) are approximately pairwise solvable and therefore they are approximately solvable. Let $\alpha_1: Q_1 \rightarrow S$ together with $x'_1:Q_1 \rightarrow X$, $\alpha_2:Q_2 \rightarrow S$ together with  $x'_2:Q_2 \rightarrow X$ and $\alpha_3:Q_3 \rightarrow S$ together with $x'_3:Q_3 \rightarrow X$ be the approximate solutions of $(*)$ after removing the first, second and third rows respectively. Consider the limit of the diagram:
	\[
	Q_i \xrightarrow{\alpha_i} S \tag{$i = 1,2,3$}
	\]
	which gives an object $Q$ together with three regular epimorphisms $p_1,p_2,p_3$ making the diagram 
	\[
	\xymatrix{
		Q \ar[d]_{p_i}\ar[dr]^{\alpha}& \\
		Q_i \ar[r]_{\alpha_i} & S
	}
	\]
	commute, where $\alpha$ is any composite $\alpha_i p_i$. Define $x_i = x'_i p_i$ for $i = 1,2,3$, then we have that $\alpha$ together with $x_1, \alpha$ together with $x_2$, and $\alpha$ together with $x_3$, are approximate solutions of the $(*)$ after removing the first, second and third row respectively. 
	Now, let $e:X \rightarrow R_0$ and $r: R_0 \rightarrow X_1 \times X_2 \times \cdots \times X_{m+1}$ be the regular epi and mono part of the regular image factorization of $(f_1,f_2,...,f_{m+1}): X \rightarrow X_1\times X_2 \times \cdots \times X_{m+1}$, and let $R$ be the $(m+1)$-ary relation  represented by $r$. Then we have 
	\begin{align*}
	(f_1 x_1 ,f_2 a_2\alpha,f_3 a_3\alpha, ...,f_{m+1} a_{m+1} \alpha) \eot{Q}& R \quad \text{and}, \\
	(f_1  a_1\alpha, f_2 x_2, f_3 a_3\alpha,...,f_{m+1} a_{m+1} \alpha ) \eot{Q}& R \quad \text{and}, \\
	(f_1 a_1 \alpha, f_2 a_2\alpha, f_3 x_3,...,f_{m+1} a_{m+1}\alpha ) \eot{Q}& R,
	\end{align*}
	which by Lemma~\ref{lem-majority-selecting}, implies that $(f_1a_1\alpha,f_2a_2\alpha,...,f_{m+1} a_{m+1}\alpha) \eot{Q} R$. Therefore, there exists $g:Q \rightarrow R_0$ making the square
	\[
	\xymatrix{
		Q \ar[d]_g \ar@{->>}[r]^\alpha \ar[r]& S \ar[d]^{(f_1a_1,f_2a_2,...,f_{m+1} a_{m+1})} \ar@{..>}[dl]_h \\
		R_0  \ar[r]_-r& X_1 \times X_2 \times \cdots \times X_{m+1}
	}
	\]
	commute. The morphism $h$ exists because $\alpha$ is a regular epimorphism. Finally, by pulling back $h$ along $e$, we get the commutative diagram:
	\[
	\xymatrix{
		Q' \ar[d]_{a'} \ar@{->>}[r]^{\alpha'} & S \ar[d]^h \ar[rd]^-{ \quad \quad (f_1a_1,f_2a_2,...,f_{m+1} a_{m+1})}&  \\
		X \ar@{->>}[r]_e & R_0 \ar[r]_-r & X_1 \times X_2 \times \cdots \times X_{m+1}
	}
	\] 
	where $a'$ is an approximate solution of the system ($*$) with approximation $\alpha'$.
\end{proof}
\noindent 
This brings us to the main theorem of this paper, which combines all of the previous results.
\begin{theorem} \label{thm-main-characterization}
	The following are equivalent for a regular category $\C$:
	\begin{enumerate}[(i)]
		\item $\C$ is a majority category.
		\item For any three reflexive relations $R,S,T$ on any object $X$ in $\C$ we have 
		\[
		R \circ (S \cap T) \geqslant (R \circ S) \cap (R \circ T)
		\]
				\item For any three reflexive relations $R,S,T$ on any object $X$ in $\C$ we have 
		\[
		R \cap (S \circ T) \leqslant (R \cap S) \circ (R \cap T)
		\]
		
		\item The Pairwise Chinese Remainder Theorem holds for $\C$.
		\item $\C$ has finite $2$-fold subobject decompositions.

	\end{enumerate}
\end{theorem}
\begin{proof}
This result follows from application of Lemma~\ref{lem-PCRT-implies-subobject-decompositions}, Theorem~\ref{thm-relationDecompositionImpiesMajority}, and Lemma~\ref{lem-majorityImpliesPCRT}. 
\end{proof}

\section{Concluding remarks}
Not all facts about varieties admitting a majority term generalize to regular (or even exact majority categories). We give two representative illustrations of where this can occur. \textit{Illustration 1:} it is well known that finitary varieties admitting a majority term are necessarily congruence distributive \cite{Pix63}, however, is not true in general that exact majority categories are congruence distributive. For a counterexample we refer the reader to Example~12.1 in \cite{GJan16}, where the author of that paper shows that the variety of distributive lattices equipped with an operation of countable arity is not even congruence modular, although it is a majority category. 
\textit{Illustration 2:} A $4$-ary near unanimity term $p(x_1,x_2,x_3,x_4)$ is a $4$-ary term satisfying, 
\begin{align*}
p(x,x,x,y) =& x, \\
p(x,x,y,x) =& x, \\
p(x,y,x,x) =& x, \\
p(y,x,x,x) =& x.
\end{align*}
 Clearly the above system of equations determines an \emph{elementary} matix $M$ of terms in the sense of \cite{ZJan04}:  
\[
M = \begin{pmatrix}
x & x & x & y & \vline & x \\
x' & x' & y' & x'& \vline & x' \\
x'' & y'' & x'' & x'' & \vline & x'' \\
y''' & x''' & x''' & x''' & \vline & x''' 
\end{pmatrix}.
\]
If we call categories  $4$-unanimous when they are strictly $M$-closed (see \cite{ZJan04}), then 4-unanimous varieties precisely those that admit a $4$-ary near unanimity term. Now, if a variety of algebras possesses such a term, then it is congruence distributive (see \cite{Mit78}). Consequently, if $\V$ is a Mal'tsev variety which admits a near unanimity term, then $\V$ admits a majority term (see \cite{Pix63}). Thus for varieties, we have the relationship 
\[
\text{Mal'tsev + 4-unanimous = Mal'tsev + majority},
\]
among these notions. This relationship extends to Barr-exact categories: if $\C$ is a Barr exact Mal'tsev 4-unanimous category, then $\C$ is majority category. However, this relationship does not extend to regular categories (see section 5 of \cite{HoePhd}). This shows that there can be relationships between matrix conditions \cite{ZJan04}, which depend on subtle exactness conditions such as every equivalence relation being effective. 
\section*{Acknowledgements}
Many thanks are due to Prof.~Z.~Janelidze, for many helpful and stimulating discussions on the topic of regular majority categories.

\bibliographystyle{plain}
\bibliography{Mike}

\end{document}